\documentclass[12pt,reqno]{article}

\usepackage[usenames]{color}
\usepackage{amssymb}
\usepackage{graphicx}
\usepackage{amscd}

\usepackage[colorlinks=true,
linkcolor=webgreen,
filecolor=webbrown,
citecolor=webgreen]{hyperref}

\definecolor{webgreen}{rgb}{0,.5,0}
\definecolor{webbrown}{rgb}{.6,0,0}

\usepackage{color}
\usepackage{fullpage}
\usepackage{float}

\usepackage{graphics,amsmath,amssymb}
\usepackage{amsthm}
\usepackage{amsfonts}
\usepackage{latexsym}
\usepackage{epsf}

\begin{document}

\theoremstyle{plain}
\newtheorem{theorem}{Theorem}
\newtheorem{corollary}[theorem]{Corollary}
\newtheorem{lemma}[theorem]{Lemma}
\newtheorem{proposition}[theorem]{Proposition}

\theoremstyle{definition}
\newtheorem{definition}[theorem]{Definition}
\newtheorem{example}[theorem]{Example}
\newtheorem{conjecture}[theorem]{Conjecture}

\theoremstyle{remark}
\newtheorem{remark}[theorem]{Remark}

\begin{center}
\vskip 1cm{\LARGE\bf 
  Mersenne Primes in Real Quadratic Fields}
\vskip 1cm
\large Sushma Palimar and Shankar B R\\
Department of Mathematical and Computational Sciences \\
National Institute of Technology Karnataka, Surathkal\\
Mangalore, INDIA.\\
\href{mailto:sushmapalimar@gmail.com}{\tt sushmapalimar@gmail.com}\\
\href{mailto:shankarbr@gmail.com}{\tt shankarbr@gmail.com}\\
\end{center}

\vskip .2 in

\begin{abstract}
The concept of Mersenne primes is studied in  real quadratic fields of  class number 1. 
Computational results are given. The field $Q(\sqrt{2})$ is studied in detail with a focus on 
representing Mersenne primes in the form $x^{2}+7y^{2}$. It is also  proved  that
 $x$ is divisible by $8$ and $y\equiv \pm3\pmod{8}$ generalizing the result of F Lemmermeyer, first proved in \cite{LS}
using Artin's Reciprocity law.
\end{abstract}

\section{Introduction}

It is well known that $a^{d}-1$ divides $a^{n}-1$ for each divisor $d$ of $n,$ and if $n=p$, a prime,
 then \begin{equation}\label{1.1}
                         a^{p}-1=(a-1)(1+a+a^{2}+...+a^{p-1})
                        \end{equation}
and if  $a^{p}-1$ is a prime, then $a=2$.\\ 
Number theorists of all persuasions have been fascinated by prime numbers of the form $2^{p}-1$ ever since 
\textit{Euclid}  used them for the construction of perfect numbers. In modern times, they are named after 
\textit{Marin Mersenne} (1588-1648).
A well known result due to Euclid is that, if $2^{p}-1$ is a prime, then $2^{p-1}(2^{p}-1)$ is perfect.
Much later $Euler$ proved the converse,  every even-perfect number has this form. \\Mersenne primes have been studied 
by amateurs as well as specialists. Mersenne primes are used in cryptography too in generating pseudorandom numbers.
 By far, the most widely used technique  for pseudorandom number generation is an algorithm first proposed by Lehmer, known 
as the linear congruential method.  It is generated by the recursion $X_{n+1}\equiv aX_{n}\pmod{M_{31}}$, where $M_{31}$
is the Mersenne prime $2^{31}-1.$ Of the more than two billion choices for $a$ only handful of multipliers are useful. 
One such value is $a=7^{5}=16807,$ which was originally designed for use in the IBM 360 family of computers\cite{WS}. \\
On March 3, 1998, the birth centenary of \textit{Emil Artin } was celebrated at the Universiteit van Amsterdam. 
The paper \cite{LS}  is based on two lectures given on the occasion. We quote 
from \cite{LS}: 
``Artin's reciprocity law is one of the cornerstones of \textit{class field theory.} 
To illustrate its usefulness in elementary number theory, we shall apply it to prove a recently observed property of
Mersenne
primes.'' The property of Mersenne primes referred to is the following:
if $M_{p}=2^{p}-1$ is prime and $p\equiv 1\pmod{3}$, then $M_{p}=x^{2}+7y^{2}$ for some integers $x,y$ and one always has 
$x\equiv 0\pmod{8}$. Also, $y\equiv \pm 3 \pmod{8}.$ This was first observed by \textit{Franz Lemmermeyer}. \\ 
Many have attempted to generalise the notion of Mersenne primes and even-perfect numbers to complex quadratic fields 
with class number 1. One reason is that they have only finitely many units. Indeed, with the exception of $Q(\sqrt{-1})$
and $Q(\sqrt{-3})$, the other seven complex quadratic fields with class number 1 have only two units: $\pm1$.
\textit{Robert Spira}  defined Mersenne primes over $Q(\sqrt{-1})$
 to give a useful definition of even-perfect 
numbers over $Z[i]$, the ring of Gaussian integers\cite{RS} . His work was continued later 
by \textit{Wayne L Mc Daniel} to give an
analogue of \textit{Euclid-Euler} Theorem over $Q(\sqrt{-1})$ and $Q(\sqrt{-3})$ \cite{WM,WL}.
In both the papers the concept of Mersenne primes is used to give a valid definition of even-perfect numbers.\\
Recently \textit{Pedro Berrizbeitia} and \textit{Boris Iskra}  studied Mersenne primes over Gaussian integers and 
Eisenstein integers \cite{BI}. The  primality of Gaussian Mersenne numbers and Eisenstein Mersenne numbers are tested using 
 biquadratic reciprocity  and
cubic reciprocity laws respectively.    
\paragraph{}In this paper the concept of Mersenne primes has been extended to  real quadratic fields 
$K=Q(\sqrt{d})$  with class number 1, so that unique factorization holds and irreducibles are always prime.
We denote the ring of integers of $K$ by  ${\cal O_{K}}$, 
\begin{displaymath}
{\cal O_{K}} =\left\{ \begin{array}{ll} Z[\sqrt{d}]      & \textrm{ if \(d \equiv\ 2,3\pmod{4}\)}\\
 Z[\frac{1+\sqrt{d}}{2}]      & \textrm{if  \(d \equiv\ 1\pmod{4}\)}
 \end{array} \right.
\end{displaymath}
Since $K$ is a Unique Factorization Domain, irreducibles  are primes  in these domains.
 Hence for any $\eta$ $\in$ $K$ 
the two factorings of $\eta$ say
\begin{center}
 $\eta  = \pi_1^{k_1}\pi_2^{k_2}...\pi_r^{k_r} $ and 
$\eta = \epsilon_1 \pi_1^{k_1}\epsilon_2 \pi_2^{k_2}...\epsilon_r \pi_r^{k_r}$
\end{center}
are considered to be one and the same, where $\epsilon_i$ are units and $\pi_i$ are irreducibles. 
\paragraph*{}
We define $M_{p,\alpha}=\frac{\alpha^{p}-1}{\alpha-1}\,$ such that
$\alpha \in \mathcal{O_{K}}$ is irreducible and $(\alpha-1)=u$ is a unit other than $\pm 1$.
Then  $M_{p,\alpha}$ may be called as  an analog of Mersenne prime if the norm of 
$M_{p,\alpha}$ namely $N(M_{p,\alpha})=N(\frac{\alpha^{p}-1}{a-1})$
is a rational prime. Condition for the irreducibility of $1+u=\alpha \in  {\cal O_{K}}$  such 
that $\alpha-1$ is a unit (other than $\pm 1$) is  derived in the next section.
For this we study the case $N(\alpha-1)=N(u)=\pm 1$ separately. We also give a list 
of such quadratic fields and a few Mersenne primes in those fields.\\ Computational results show that, among real quadratic 
 fields, Mersenne primes in $Q(\sqrt{2})$ have a definite structure. The special property of the usual 
Mersenne primes observed by  \textit{Franz Lemmermeyer} and proved in \cite{LS} seems to admit a generalisation to 
Mersenne primes over $Q(\sqrt{2})$. This property appears to be special only to   $Q(\sqrt{2})$.
Some interesting properties of Mersenne primes 
 and recent  primality tests to check the primality of   Mersenne numbers in $Q(\sqrt{2})$
are given.
 Also, the usual Mersenne primes  given by $M_{p}=2^{p}-1,$ can be  obtained from the field  $K\,=\,Q(\sqrt{2})$
 without altering the conditions on   $M_{p,\alpha}$. Below we consider various cases under which $\alpha$ is irreducible.

\begin{theorem}\label{th1}
  Let $d\equiv 2,3 \pmod {4}$ and $N(\alpha-1)=-1$. Then $\alpha$ is irreducible if and only if $d=2$ and 
$u\in \{1+\sqrt{2},\,1-\sqrt{2},\,-1+\sqrt{2},\,-1-\sqrt{2} \}$. 
 \end{theorem}
\begin{proof}
 Let  $\alpha$ be  irreducible and $u=a+b\sqrt{d}$. Then $\alpha=(a+1)+b\sqrt{d}$. Hence,
$N(\alpha)= (a+1)^{2}-2b^{2}=N(u)+2a+1= 2a$. 
Since $\alpha$ is irreducible, $2a$ should be a rational prime. Hence $a=\pm1$. With $a=1$,  $u=1+b\sqrt{d}$ and 
$ N(u)=-1= 1-b^{2}d$. i.e., $b^{2}d=2$. Since  $d$ is square-free, $d=2$  and $b=\pm1$.  Similarly 
 with $a=-1$, we get $b=\pm1$ and $d=2$. \\
 Conversely let $d=2$ and $u=a+b\sqrt{2}$ be any unit in $Q(\sqrt{2})$. Then  $\alpha=(a+1)+b\sqrt{2}$ 
  and $N(\alpha)=(a+1)^{2}-2b^{2}=a^{2}-2b^{2}+2a+1=N(u)+2a+1=2a$, is a rational prime, if and only if $a=\pm1$.
As before, we get $b=\pm 1.$\\
Hence, different choices  of $u$ for which $\alpha$ is irreducible  are respectively,
$1+\sqrt{2},\,1-\sqrt{2},\,-1+\sqrt{2}\,$ and $-1-\sqrt{2}$. As $1+\sqrt{2}$ is the fundamental unit, 
these values are,
$u,-u^{-1},u^{-1},-u$. Corresponding $\alpha$ values are, $2+\sqrt{2},\,2-\sqrt{2},\,\sqrt{2}$ and $-\sqrt{2}$.\\
\end{proof}
Since $2-\sqrt{2}$ and $-\sqrt{2} $ are the  conjugates of $2+\sqrt{2}$ and $\sqrt{2}$ respectively, we compute $M_{p,\alpha}$
with $\alpha=2+\sqrt{2}$ and $\sqrt{2}$.\\
For $\alpha=2+\sqrt{2}$ a few Mersenne primes in $Q(\sqrt{2})$ are given below:
 
\begin{center}
 \textit{Table (1)}\\
\begin{tabular}{|l|l|l|l|}                                   \hline
  $p$   &$M_{p,\alpha}$                        &$N(M_{p,\alpha})$          \\ \hline
$2$      &$3+\sqrt{2}$                                &$7$                     \\ \hline
$3$       &$9+5\sqrt{2} $                                &$31$                     \\ \hline
$5$         &$97+67\sqrt{2}$                              &$431$  		     \\ \hline
$7$           &$1121+791\sqrt{2}$                               &$5279$                       \\ \hline
$11$           &$152193+107615\sqrt{2}$                            &$732799$ 	               \\ \hline         
\end{tabular}
\end{center}
The  next Mersenne primes are found at \\$p=73$, with \\$N(M_{p,\alpha}) = 851569055172258793218602741480913108991$,\\
$p=89$  with  $N(M_{p,\alpha})=290315886781191681464330388772329064268797313023,$\\ 
$p=233$ with $N(M_{p,\alpha}) = 18060475427282023033368001231166441784737806891537$\\
$806547065314167911959518498581747712829157156517940837234519177963497324543.$
\paragraph{} With $\alpha = \sqrt{2}$,
 $M_{p,\alpha}=\frac{(\sqrt{2})^{p}-1}{\sqrt{2}-1}$. 
Thus, $N(M_{p,\alpha})=2^{p}-1$, giving all the usual Mersenne numbers. 
\begin{theorem}\label{th2}
  Let \textrm{ \(d \equiv\ 1\pmod{4}\)} and  $\alpha-1=u=\frac{a+b\sqrt{d}}{2}$ be a unit such that,
$N(u)=N(\alpha-1)=-1$. 
Then, $\alpha$ is irreducible, if and only if,   $a$ is a rational prime and $b$ is some  odd integer. 
\end{theorem}
\begin{proof} By hypothesis, $ N(\alpha)= \frac{(a+2)^{2}-db^{2}}{4} = a$ 
since $N(u)=-1$. For $\alpha$ to be irreducible $a$ should be an odd  rational prime.\\ 
Indeed if $a=2$ then $u=\frac{2+b\sqrt{d}}{2}$ and $N(u)=\frac{4-db^{2}}{4}= -1 $
$\Rightarrow b^{2}d=8.$ This is impossible since $d \equiv 1 \pmod{4}.$
 Since  \textrm { \(d \equiv\ 1\pmod{4}\)}  it is clear that 
$b$ is some odd integer.  Thus, the analogs of Mersenne primes  are defined for 
\textrm{ \(d \equiv\ 1\pmod{4}\)}  whenever   units are of the form  $u= \frac{p+(2n+1)\sqrt{d}}{2}$,
where $n \in Z $ and $p$ is an odd rational prime.\\
The  converse is straightforward since the norm of $\alpha$ is $a=p$, a rational prime by assumption.
\end{proof}
The Table below shows the values of  \textrm{ \(d \equiv\ 1\pmod{4}\)} $\forall d<500$ for which the class number is 1,
  $N(u)=-1$ and  $\alpha$ is irreducible.
\begin{center}
\textit{Table (2)} \\
\begin{tabular}{|l|l|l|l|}   \hline 
$Q(\sqrt{d})$              &$u$                           &$\alpha$                             & $N(\alpha)$    \\ \hline
$Q(\sqrt{13})$          &$(\frac{3+\sqrt{13}}{2})$         &$(\frac{5+\sqrt{13}}{2})$             &$3$           \\ \hline
$Q(\sqrt{29})$          &$(\frac{5+\sqrt{29}}{2})$          &$(\frac{7+\sqrt{29}}{2})$           &$5$      \\ \hline
$Q(\sqrt{53})$          &$(\frac{7+\sqrt{53}}{2})$          &$(\frac{9+\sqrt{53}}{2})$           &$7$      \\ \hline
$Q(\sqrt{149})$          &$(\frac{61+5\sqrt{149}}{2})$          &$(\frac{63+5\sqrt{149}}{2})$           &$61$      \\ \hline
$Q(\sqrt{173})$          &$(\frac{13+\sqrt{173}}{2})$          &$(\frac{15+\sqrt{173}}{2})$           &$13$      \\ \hline
$Q(\sqrt{293})$          &$(\frac{17+\sqrt{293}}{2})$          &$(\frac{19+\sqrt{293}}{2})$            &$17$     \\ \hline
\end{tabular} 
 \end{center}
 
As an illustration:
\begin{center}
\textit{Table (3)} for $K=Q(\sqrt{13}), u=\frac{3+\sqrt{13}}{2}$\\
\begin{tabular}{|l|l|l|l|}   \hline
$p$         &$N(M_{p,\alpha})$                                                         \\ \hline
$5$         &$1231$                                                                    \\ \hline 
$7$         &$25117$                                                         \\ \hline 
$11$	    &$9181987$	                                                \\ \hline                                                       
$19$	     &$1098413907397$                       \\ \hline                                                                                                                       
\end{tabular} 
\end{center}

The next Mersenne prime  is found at $p=41$.
 \subparagraph*{}
 \begin{theorem}\label{th3} Let \textrm{ \(d \equiv\ 2,3\pmod{4}\)} and $u=a+b\sqrt{d}$ 
be a unit, such that $N(u)=1$, then $\alpha$ is always reducible.
\end{theorem}
\begin{proof}
By hypothesis, $\alpha=(a+1)+b\sqrt{d}$ and  $ N(\alpha)= 2(1+a)$, which is  prime only if $a=0$,
 which contradicts $N(u)=1$. Hence $\alpha$ is not irreducible.
\end{proof}
\begin{theorem}\label{th4}Let \textrm{ \(d \equiv\ 1\pmod{4}\)} and  $u=\frac{a+b\sqrt{d}}{2}$ be a unit such that $N(u)=1$. 
Then, $\alpha$ is irreducible, if and only if,  $a+2$ is a rational prime  and $b$ is some odd integer.
\end{theorem}
\begin{proof}
By hypothesis, $ N(\alpha)= \frac{(a+2)^{2}-db^{2}}{4} = a+2$, since $N(u)=1$. For $\alpha$ to be irreducible
 $a+2$ should be a rational prime. Clearly $a \neq 0$. Hence $a+2$ is an odd rational prime. 
Hence, \textrm{ \(a^{2} \equiv\ 1\pmod{4}\)}. Since  \textrm { \(d \equiv\ 1\pmod{4}\)}  it is clear that 
$b$ is some odd integer.  Thus, the analogs of Mersenne primes  are defined for 
\textrm{ \(d \equiv\ 1\pmod{4}\)}  whenever   units are of the form  $u= \frac{a+(2n+1)\sqrt{d}}{2}$,
where $n \in Z $ and $a+2$ is an  odd rational prime.\\
Converse is straightforward as in Theorem \ref{th2}.
\end{proof}
The Table below shows the values of  \textrm{ \(d \equiv\ 1\pmod{4}\)} $\forall d<500$ for which
the class number is 1,
  $N(u)=1$ and  $\alpha$ is irreducible.
\subparagraph{}
\begin{center}
\label{ta3}\textit{Table(4)} \\ 
 \begin{tabular}{|l|l|l|l|}   \hline 
$Q(\sqrt{d})$              & $u$                           &$\alpha$                             & $N(\alpha)$    \\ \hline
$Q(\sqrt{21})$        &$\frac{5+\sqrt{21}}{2}$           &$\frac{7+\sqrt{21}}{2}$           &$7$            \\  \hline
$Q(\sqrt{77})$        &$\frac{9+\sqrt{77}}{2}$           &$\frac{11+\sqrt{77}}{2}$           &$11$        \\ \hline
$Q(\sqrt{93})$        &$\frac{29+3\sqrt{93}}{2}$         &$\frac{31+3\sqrt{93}}{2}$           &$31$         \\ \hline
$Q(\sqrt{237})$       &$\frac{77+5\sqrt{237}}{2}$        &$\frac{79+5\sqrt{237}}{2}$         &$79$       \\ \hline   
$Q(\sqrt{437})$       &$\frac{21+\sqrt{437}}{2}$         &$\frac{23+\sqrt{437}}{2}$         &$23$        \\ \hline
$Q(\sqrt{453})$       &$\frac{149+7\sqrt{453}}{2}$       &$\frac {151+7\sqrt{453}}{2}$       &$151$     \\ \hline
\end{tabular}      
\end{center}
As an illustration we consider the following table.
\begin{center}
\textit{Table (5)} for $K=Q(\sqrt{21}),\,u=\frac{5+\sqrt{21}}{2}$ \\
\begin{tabular} {|l|l|l|l|}   \hline
 $p$       &$N(M_{p,\alpha})$   \\ \hline
 $17$       &$223358425353211 $ \\ \hline
 \end{tabular} 
 \end{center}
The next Mersenne prime is found at $p=47$.\\
Similar calculations are obtained for $Q(\sqrt{77})$, the fundamental unit is $u=\frac{9+\sqrt{77}}{2}$ 
and $\alpha=\frac{11+\sqrt{77}}{2}.$
\begin{center}
\textit{Table (6)} for $K=Q(\sqrt{77}), \,u=\frac{9+\sqrt{77}}{2}$  \\
\begin{tabular}{|l|l|l|l|}   \hline

$p$         &$N(M_{p,\alpha})$                \\ \hline
$2$         &$23 $                    \\ \hline 
$7$         &$10248701 $           \\ \hline 
\end{tabular}
\end{center}
The next Mersenne prime is found at $p=71$.\\
The values of $u$ for the Tables (2) and (4) are taken from \cite{cohen}.
\paragraph{Remarks}
\begin{enumerate}
\item In \textit{Table (2)} and \textit{Table (4)} above, we have chosen only the fundamental unit
$u$ in $Q(\sqrt{d})$. However, it is possible that, $\alpha=1+u$ is not irreducible with $u$ as fundamental unit and yet 
$\alpha^{\prime}=1+u^{\prime}$ is irreducible for some other unit $u^{\prime}$ in $Q(\sqrt{d})$.\\
 As an illustration we consider $Q(\sqrt{5})$. Here, $u=\frac{1+\sqrt{5}}{2}$ is the fundamental unit. But, 
$\alpha=1+u=\frac{3+\sqrt{5}}{2}=u^{2}$ is again a unit! However, with $u^{\prime}=u^{2}=\frac{3+\sqrt{5}}{2}$, we get
$\alpha^{\prime}=1+u^{\prime}=\frac{5+\sqrt{5}}{2}$ and $N(\alpha^{\prime})=5$, so $\alpha^{\prime}$ is irreducible. Another
choice is $u^{5}=u^{''}=\frac{11+5\sqrt{5}}{2}$ and $\alpha^{''}=1+u^{''}=\frac{13+5\sqrt{5}}{2}$ 
is irreducible since $N(\alpha^{''})=11$.
\item Theorems \ref{th1} and \ref{th3} imply the following: Among all fields $Q(\sqrt{d})$, $d\equiv 2,3 \pmod{4}$
$Q(\sqrt{2})$ is the only field where $ \alpha=1+u$ is irreducible. There are essentially  only two choices for $\alpha$,
namely $\sqrt{2}$ and $2+\sqrt{2}$. 
\end{enumerate}\
Similar to usual Mersenne primes in $Z$, quadratic Mersenne norms have the following properties:
\paragraph{Properties of $N(M_{p,\alpha})$:\\}
\begin{enumerate}
\item If $N(M_{n,\alpha})$ is prime, then $n$ is prime.
\item The sequence $\{N(M_{n,\alpha})\}_{n=1}^{\infty}$ is an increasing sequence of integers that starts at $1.$
\item If $d$ divides $n$ then $M_{d,\alpha}$ divides $M_{n,\alpha}$ in $Q(\sqrt{d})$ and $N(M_{d,\alpha})$
 divides $N(M_{n,\alpha})$.
\item If $d$ and $n$ are relatively prime then $M_{d,\alpha}$ is relatively prime to $M_{n,\alpha}$   in 
$Q({\sqrt{d}})$ 
and $N(M_{n,\alpha})$ is relatively prime to $N(M_{d,\alpha})$.
\end{enumerate}
Experimental evidence shows that Mersenne primes are  sparse in $Q(\sqrt{d})$ for \textrm{ \(d \equiv\ 1\pmod{4}\)}.
Some interesting properties of Mersenne primes in $Q(\sqrt{2})$ are given below.
\paragraph{Properties of Mersenne primes in $Q(\sqrt{2})$\\}
\begin{enumerate}
\item Since $\alpha=1+u\,=\,2+\sqrt{2}\,=\,u\sqrt{2}$, where $u$ is the fundamental unit, we have\\
$ \alpha^{n}=a_{n}+b_{n}\sqrt{2}=u^{n}(\sqrt{2})^{n}\,$, for any integer $n>0$
 and $a_{n},b_{n}\in Z$.  A small calculation also reveals that,
 \begin{displaymath}
\alpha^{n}= \left\{ \begin{array}{ll} 
(2^{\frac{n-1}{2}}\sqrt{2})u^{n} \quad if \; n \; is \;\, odd \\
2^{\frac{n}{2}}u^{n} \quad if \; n \; is \;\, even \\
  \end{array} \right.\\
\end{displaymath}
which is 
\begin{displaymath}
\alpha^{n}= \left\{ \begin{array}{ll} 
(2^{\frac{n-1}{2}}\sqrt{2}) (v_{n}+w_{n}\sqrt{2})\quad if \;\; n \;\, is \;\; odd;\quad w_{n}, v_{n}\in Z\\
2^{\frac{n}{2}}(v_{n}^{\prime}+w_{n}^{\prime}\sqrt{2}) \quad if \;\; n \;\, is \;\, even;\quad  v_{n}^{\prime},w_{n}^{\prime}\in Z\\
  \end{array} \right.\\
\end{displaymath}
It can be noted that, $w_{n}$, the coefficient of $\sqrt{2}$ in $u^{n}$ is odd if $n$ is odd.
 Also, $\,2^{\frac{n+1}{2}}w_{n}=a_{n}$ 
and $2^{\frac{n-1}{2}}v_{n}=b_{n}$.\\
And, $w_{n}^{\prime}$, the coefficient of $\sqrt{2}$
in $u^{n}$ is even if $n$ is even.  Also, 
 $2^{\frac{n}{2}}v_{n}^{\prime}=a_{n}\;$
 and $\;2^{\frac{n}{2}}w_{n}^{\prime}=b_{n}$.\\  
For $n$ odd, $ N(u)^{n}=-1$, so \[N(\alpha^{n})=N(2^{\frac{n-1}{2}}\sqrt{2})N(u)^{n}=
N(2^{\frac{n-1}{2}})N(\sqrt{2}) (-1)^{n}=2^{n-1}(-2)(-1)\,=\,2^{n}\]
For $n$ even,  $ N(u)^{n}=1$, and  \[N(\alpha^{n})=N(2^{\frac{n}{2}})N(u)^{n}=N(2^{\frac{n}{2}})( 1)=2^{n} \]
\item For any odd prime $p,$ let $\alpha^{p}= (2^{\frac{p-1}{2}}\sqrt{2}) (v_{p}+w_{p}\sqrt{2})$.\\ 
Then, 
\[ N(\alpha^{p}-1)= (2^{\frac{p+1}{2}}w_{p}-1)^{2}-2(2^{\frac{p-1}{2}}{v_{p}})^{2}\]
\[=(2^{p+1}{w_{p}}^{2}+1-2^{\frac{p+3}{2}}w_p)-2^{p}{v_p}^{2} \]
\[=2^{p}(2{w_{p}}^{2}-{v_p}^{2})-2^{\frac{p+3}{2}}w_p +1\]
\[= 2^{p} -2^{\frac{p+3}{2}}w_p +1\]
  But,
 \[N(M_{p,\alpha})=\, 2^{\frac{p+3}{2}}w_{p}-2^{p}-1 \]     
     
\item As already noticed, $a_{p}$ has a factor of $2^{\frac{p+1}{2}}$.
Hence, $2a_{p}\equiv 0\pmod{4}$. 
This further implies that, $N(M_{p,\alpha})\equiv -1\pmod{4}$ for $p\geq2$ 
and $N(M_{p,\alpha})\equiv -1\pmod{8}$ for $p>2.$\\
The next three properties are consequences of quadratic reciprocity, and $\left(\frac{\cdot}{\cdot}\right)$ 
denotes the $Legendre$ symbol. 
\item 
 Let $p$ be an odd prime. If $p\equiv \pm 1 \pmod{8}$, then
\[ 2^{\frac{p+3}{2}}=2^{2}2^{\frac{p-1}{2}}\equiv 4\pmod{p}\]
If $p\equiv \pm 3 \pmod{8}$, then  \[ 2^{\frac{p+3}{2}}=2^{2}2^{\frac{p-1}{2}}\equiv -4\pmod{p}\]
Combining the above we get, 
\begin{displaymath}
N(M_{p,\alpha}) \equiv \left\{ \begin{array}{ll} 4w_{p}-3  \pmod{p}     & \textrm{ if \(p \equiv\ \pm 1\pmod{8}\)}\\
 -4w_{p}-3 \pmod{p}    & \textrm{if  \(p \equiv\ \pm 3\pmod{8}\)}
 \end{array} \right.\\
\end{displaymath}
\item  
If $N(M_{p,\alpha})$ is a rational prime and $q$ is any other prime
 then \begin{displaymath}
\left(\frac{N(M_{p,\alpha})}{q}\right)\left(\frac{q}{N(M_{p,\alpha})}\right)= \left\{ \begin{array}{ll} 
 1 \quad if \quad q\equiv 1\pmod{4},\\
 -1 \quad if \quad q\equiv 3\pmod{4},\\
 \end{array} \right.\\
\end{displaymath}
\item If $N(M_{p,\alpha})$ is a rational prime then $\left( \frac{2}{N(M_{p,\alpha})}\right)=1$ since  
$N(M_{p,\alpha})\equiv -1\pmod{8}.$\\  Hence,
\mbox{$\sqrt{2} \in F_{N(M_{p,\alpha})}$}, the finite field with \mbox{$N(M_{p,\alpha})$} elements.
\end{enumerate}

 \section{Testing for primality}

Several primality tests are available and some are speciallly designed for special numbers, an example  
being the famous Lucas-Lehmer Test for the usual Mersenne primes. 
We show that the generalised Mersenne numbers of $Q(\sqrt{2})$ can be put in a special form, so that, 
recent primality tests can be used to determine 
whether they are prime.
Now,
 \[N(M_{p,\alpha})= \,2^{\frac{p+3}{2}}w_{p}-2^{p}-1,\]
or,
\[N(M_{p,\alpha})=2^{\frac{p+3}{2}}(w_{p}-2^{\frac{p-3}{2}})-1\]
Since $w_{p}$ is odd, $(w_{p}-2^{\frac{p-3}{2}})$ is odd for $p>3.$\\
For  $ p>3,$ \[ N(M_{p,\alpha})=h.2^{\frac{p+3}{2}}-1, \quad  where \quad h=(w_{p}-2^{\frac{p-3}{2}}), \quad odd. \] 
\subsection*{}  An algorithm to test the primality of numbers of the form  
$h\cdot2^{n}\pm1$, for any odd integer $h$ such that, $h\neq 4^{m}-1$ for any $m$ is described in\cite{BW}. 
It can be noted that, $h$ is not equal to $4^{m}-1$ in $M_{p,\alpha}$ for any $m$.
 Hence, $Bosma's$ algorithm \cite{BW} can be used to test 
the primality of $M_{p,\alpha}$.
Also, \cite{YT}  describes an algorithm to test the primality of
 numbers of the form $h\cdot2^{n}-1$ for an odd integer $h$
using Elliptic curves, which is the elliptic curve version of the Lucas-Lehmer-Riesel primality test.
Thus the above two tests can be used to test the primality of $N(M_{p,\alpha})$. 
\section{Primes of the form \texorpdfstring{$x^{2}+7y^{2}$}{x2+7y2}} 
\subparagraph{}
The problem of representing a prime number by the form $x^{2}+ny^{2}$, where $n $ is any fixed positive integer
dates back to Fermat.
This question was best answered by $Euler$ who spent $40$ years in  proving Fermat's theorem and thinking about 
how they can be generealised, he  proposed some conjectures concerning $p=x^{2}+ny^{2}$, for $n>3$.
These remarkable
conjectures, among other things, touch on quadratic forms and their composition, genus theory, 
cubic and biquadratic reciprocity. Refer \cite{DC} for a thorough treatment.\\
Euler became intensely interested in this question in the early 1740's and he mentions numerous examples in his letters
to Goldbach. One among several of his conjectures stated in modern notation is 
\begin{displaymath}
 \left(\frac{-7}{p}\right)=1\Longleftrightarrow p\equiv 1,9,11,15,23,25 \pmod{28}
\end{displaymath}
The following lemma gives necessary and sufficient condition for a number $m$ to be 
represented by a form of discriminant $D$.
\begin{lemma}
 Let $D\equiv 0,1\pmod{4}$ and $m$ be an integer relatively prime to $D.$ Then $m$ is properly 
represented by a primitive form of discriminant $D$ if and only if $D$ is a quadratic residue modulo $m$. 
\end{lemma}
As a corollary, we have the following:
\begin{corollary}
 Let $n$ be an integer and $p$ be an odd prime not dividing $n.$ Then $\left(\frac{-n}{p}\right)=1$ if and only if $p$
is represented by a primitive form of discriminant $-4n.$
\end{corollary}
In 1903, Landau proved a conjecture of Gauss:\\
Let $h(D)$ denote the number of classes of primitive positive definite forms of discriminant $D$, i.e.,
$h(D)$ is equal to the number of reduced forms of discriminant $D$. 
\begin{theorem}
 Let $n$ be a positive integer. Then \begin{center}
                                      $h(-4n) =1 \Leftrightarrow n= 1,2,3,4\,\,or\, 7$.
                                     \end{center}
\end{theorem}
One may note that, $x^{2}+ny^{2}$ is always a reduced form with discriminant $-4n.$\\
In this paper we consider the  case  $n=7$ and  represent $N(M_{p,\alpha})$ 
in the form $x^{2}+7y^{2}$ whenever  $M_{p,\alpha}$ is a Mersenne prime in $Q(\sqrt{2})$.\\
$x^{2}+7y^{2}$ is the only reduced form of discriminant $-28$, and it follows that
\begin{displaymath}
 p\,= \, x^{2}+7y^{2} \Longleftrightarrow p\equiv 1,9,11,15,23,25\pmod{28}
\end{displaymath}
for primes $p\neq 7$.
\subparagraph{}
 The special  property of the usual Mersenne primes over $Z$ referred 
to in the beginning 
\cite{LS} has the following generalisation over $Q(\sqrt{2})$: 
\begin{theorem}\label{main-th}
  If $N(M_{p,\alpha})$ is a rational prime, with $\alpha=2+\sqrt{2}$, then $N(M_{p,\alpha})$ is always a quadratic residue  $\pmod{7}$, and hence it can be 
written as $x^{2}+7y^{2}$. Also, $x$ is  divisible 
by $8$, and  $y \equiv \pm 3 \pmod{8}$.
\end{theorem}

  The detailed proof is given in \cite{thes}, using Artin's reciprocity law.\\
Here we prove the theorem in two stages:
first we show that $N(M_{p,\alpha})$ is always a quadratic residue  $\pmod{7}$.
Next, we  give an outline of the proof that $x$ is  divisible 
by $8$, and  $y \equiv \pm 3 \pmod{8}$.  \\
The first few 
Mersenne primes in $Q(\sqrt{2})$ with $\alpha=2+\sqrt{2},$ as well as the representations of their norms
as $x^{2}+7y^{2}$ is given below.
\begin{center}
\textit{Table (7)}\\
\begin{tabular}{|l|l|l|l|}                                   \hline
  $p$           &$M_{p,\alpha}$   &$x^{2}+7y^{2}$       \\ \hline
$5$               &$431$  		&$16^{2}+7\cdot5^{2}$                 \\ \hline
$7$            &$5279$                &$64^{2}+7\cdot13^{2}$                    \\ \hline
$11$             &$732799$ 	       &$856^{2}+7\cdot3^{2}$                     \\ \hline         
\end{tabular}
\end{center}
\subparagraph{} For $p=73$, \\ $N(M_{p,\alpha})=851569055172258793218602741480913108991$ = \\
$(28615996544447548272)^{2}+7\cdot(2161143775888286749)^{2}$\\ 
 For $p=\,89$,  \\ $N(M_{p,\alpha}) = 290315886781191681464330388772329064268797313023$ =\\
$(363706809248848497658560)^{2}+7\cdot(150253711001099458172317)^{2}$\\
For $p=\,233$,  \\ $N(M_{p,\alpha})$ =
$1806047542728202303336800123116644178473780689153780654706531416$\\$7911959518498581747712829157156517940837234519177963497324543.$\\  
The corresponding representation is \\
$(86527345603258677818378326573842407929031070590321223524182584)^{2} +$\\
$7\cdot(38865140256563104639356290982349294477380709218952585423373629)^{2}$
\paragraph{} We now show that, if $N(M_{p,\alpha})$ is a prime then, $N(M_{p,\alpha})$ can be written as $x^{2}+7y^{2}$.\\
Since $N(M_{p,\alpha})= 2^{\frac{p+3}{2}}w_{p}-2^{p}-1,$ representing a prime in the form $x^{2}+7y^{2}$
depends on $w_{p}$. Now, we find the values of $v_{p}$ and $w_{p} \pmod{7}$.\\
As we know, for any odd  $n$, $N(u^{n})\,=\,v_{n}^{2}-2 w_{n}^{2}\,=-1$. \\
If $u^{n}=v_{n} +w_{n}\sqrt{2}$, then $v_{n}$ and $w_{n}$ satisfy the following recursions:\\
  $v_{n+1}=v_{n}+2w_{n}$ and $w_{n+1}=v_{n}+w_{n}$, with  initial conditions: $v_{1}=1,\,w_{1}=1$.\\
The above recursions can be used to show that $v_{n}$ and $w_{n}$ satisfy the following:
\begin{equation*}
 v_{n+2}=3v_{n}+4w_n \,;\, w_{n+2}=2v_n+3w_n
\end{equation*}
\begin{equation*}
 v_{n+3}=7v_n+ 10w_n \,;\, w_{n+3}=5v_n+10w_n
\end{equation*}
\begin{equation*}
v_{n+4}=17v_n+24w_n \,;\,w_{n+4}=12v_n+17w_n
\end{equation*}
\begin{equation*}
v_{n+5}=41v_n+58w_n \,;\, w_{n+5}=29v_n+41w_n
\end{equation*}
\begin{equation*}
v_{n+6}=99v_n+140w_n \,;\, w_{n+6}=70v_n+99w_n
\end{equation*}
From the above one may also easily obtain the following congruences:
\begin{equation*}
 \{v_{6k+1}\}\equiv1 \pmod{7} \; ;\;\{w_{6k+1}\}\equiv1 \pmod{7}
\end{equation*}
\begin{equation*}
 \{v_{6k+5}\}\equiv6 \pmod{7}\;;\; \{w_{6k+5}\}\equiv1 \pmod{7}
\end{equation*}
Since only odd prime powers greater than $3$ are considered, we have listed only
 the congruences for indices congruent to $\pm 1 \pmod{6}$.\\
Hence, \begin{equation} \label{eq2}
 N(M_{p,\alpha})=2^{\frac{p+3}{2}}w_{p}-2^{p}-1\equiv 2^{\frac{p+3}{2}}- 2^{p}-1 \pmod{7}
\end{equation}
Let us solve equation (\ref{eq2}) for $p>3$.\\
If $p=3k+1$ then, $2^{p}\equiv 2 \pmod{7}$ and $2^{\frac{p+3}{2}}\equiv 4 \pmod{7}$,  
so\[ N(M_{p,\alpha})\equiv 1 \pmod{7}.\]
If $p=3k+2$ then, $2^{p}\equiv 4 \pmod{7}$ and $2^{\frac{p+3}{2}}\equiv 2 \pmod{7}$,
 so \[N(M_{p,\alpha})\equiv 4\pmod{7}.\]
Thus in both cases $N(M_{p,\alpha})$ can always be represented as $x^{2}+7y^{2}$.
\paragraph{}To prove   Theorem \ref{main-th} we need the following lemma.
\begin{lemma}
 If $N(M_{p,\alpha})$ is a 
rational prime and $N(M_{p,\alpha})=x^{2}+7y^{2}$, then $x\equiv 0 \pmod{4}$, and $y\equiv \pm 3\pmod{8}.$ 
\end{lemma}
\begin{proof}
From the previous discussion, we know \begin{equation} \label{lem-eq}
                                       N(M_{p,\alpha})=x^{2}+7y^{2}
                                      \end{equation}

But, $N(M_{p,\alpha})=2^{\frac{p+3}{2}}w_{p}-2^{p}-1.$
Clearly we may take $p>6$. So either $p=6k+1$ or $p=6k+5$.\\
If $p=6k+1$, then 
\begin{equation} \label{lem-eq1}
N(M_{p,\alpha})=2^{\frac{6k+4}{2}}w_p-2^{6k+1}-1\, \equiv -1 \equiv 7\pmod{8} \end{equation}
If $p=6k+5$,   then also,
\begin{equation} \label{lem-eq2}N(M_{p,\alpha})=2^{\frac{6k+5}{2}}w_p-2^{6k+5}-1\equiv 7\pmod{8}\end{equation}
But RHS of equation (\ref{lem-eq}) is
$x^{2}+7y^{2}$. We show that $x$ must be even and $y$ odd.\\ For, if $x$ is odd and $y$ is even, then $x^{2}\equiv 1 \pmod{8}$
 and
either $y^{2}\equiv 0 \pmod{8}$ or $y^{2}\equiv 4 \pmod{8}$. If  $y^{2}\equiv 0 \pmod{8}$, 
then $x^{2}+7y^{2}\equiv 1\pmod{8}$
contradicting equations (\ref{lem-eq1}) and (\ref{lem-eq2}); and if \(y^{2}\equiv 4\pmod{8}\), 
then $x^{2}+7y^{2}\equiv1+7.4\equiv 5\pmod{8}$,
 again contradicting equations (\ref{lem-eq1}) and (\ref{lem-eq2}).
Thus $x$ is even and $y$ is odd. In this case, by equation (\ref{lem-eq})  \[7\equiv x^{2}+7y^{2}\equiv x^{2}+7\pmod{8},\]
since \(y^{2}\equiv 1\pmod{8}\) and so, $x^{2}\equiv 0 \pmod{8}$ implying $x\equiv 0\pmod{4}$.\\
We now prove that $y\equiv \pm3\pmod{8}$:\\
 Let $p\equiv 1\pmod{6}$. From  equation (\ref{lem-eq1}) \[N(M_{p,\alpha})=x^{2}+7y^{2} =2^{\frac{6k+4}{2}}w_p-2^{6k+1}-1\]  
Reducing modulo $16$, we get $N(M_{p,\alpha})\equiv -1\pmod{16}$. But $N(M_{p,\alpha})=x^{2}+7y^{2}$, and 
$x\equiv 0\pmod{4}$. Hence, $7y^2\equiv -1\pmod{16}$, yielding $y^{2}\equiv 9\pmod{16}$. This proves that,
 $y\equiv \pm3\pmod{8}$. The same result follows from equation(\ref{lem-eq2}) when $p\equiv 5\pmod{6}$.
\end{proof}
\begin{underline}
 {Outline of the Proof of Theorem} \ref{main-th}:
\end{underline}
We now show that $x\equiv 0\pmod{8}$.
Virtually, the proof given in \cite{LS} carries over word-for-word, and so, we merely give an outline. 
All details and notation are 
as in \cite{LS}. By definition,  $N(M_{p,\alpha})=\frac{(2+\sqrt{2})^p-1}{1+\sqrt{2}}. \frac{(2-\sqrt{2})^p-1}{1-\sqrt{2}}.$
 Denote the two factors on the right by $v_p$ and $\bar v_p$. 
It is easy to see that $v_p$ and $\bar v_p$ are both totally 
positive. We compute the Artin symbols of $v_{p}Z_E$ and $\bar v_{p}Z_E$, and show that 
they are both trivial. We need to consider only two cases: $p\equiv 1\pmod 6$ and
$p\equiv 5\pmod 6$.\\
Since $\sqrt{2}\equiv 3,4\pmod{7}$, by taking $\sqrt{2}=4$ in $v_p$ and $\sqrt{2}=3$ in $\bar v_p$,  
a straightforward computation shows that, $v_p\equiv 1\pmod{7}$ and $\bar v_p\equiv 1\pmod{7}$. This completes the proof. 
\paragraph{Acknowledgement\\}
The authors gratefully acknowledge the kind help and encouragement given by Prof C S Dalawat
during their visit to Harish-Chandra Research Institute, Allahabad, INDIA, in December 2011.
 His lucid explanation of Artin Reciprocity and short introduction to Class field theory
was of great help.


\begin{thebibliography}{99}
\bibitem{cohen} Henri Cohen:\textit{A Course in Computational Algebraic Number Theory}, Springer-Verlag, USA, 1993
\bibitem{BW} Wieb Bosma: \textit{Explicit Primality Criteria for} $h.2^{k}\pm1$, Mathematics of Computation, Vol.61, No.203,1993, pp. 97-109.
\bibitem{BI} Pedro Berrizbeitia, Boris Iskra: \textit{Gaussian and Eisenstein Mersenne primes}, Mathematics of Computation, vol.79, no.271, July 2010, pp.1779-1791.
\bibitem{LS} Hendrik  W Lenstra, Peter Stevenhagen: \textit{Artin Reciprocity and Mersenne Primes}, NAW 5/1. NR.1, maart 2000, pp.44-54.
\bibitem {WM} Wayne L McDaniel:  \textit{Perfect Gaussian Integers}, Acta Arithmetica, vol.XXV, 1974, pp.137-144.
\bibitem{WL} Wayne L McDaniel:  \textit{An Analogue in certain Unique Factorization Domains of the Euclid-Euler theorem on Perfect numbers}, International Jl. of Math. and Math. Sci., Vol. 13, no.1, 1990, pp.13-24.
\bibitem {RS} Robert Spira: \textit{The Complex Sum of Divisors}, The American Mathematical Monthly, Vol. 68, No.2, 1961, pp.120-124
\bibitem {WS} William Stallings: \textit{Cryptography and Network Security: Principles and Practices}, Prentice Hall India, 2006, pp. 221-222.
\bibitem {YT} Yu Tsumura: \textit{Primality tests for $2^{k}n-1$ using Elliptic Curves}, arXiv:0912.5279v1 [math.NT] Dec, 2009.
\bibitem {DC} David A Cox: \textit{Primes of the form $x^{2}+ny^{2}$ }, Wiley -Interscience Publication, 1989.
\bibitem{thes} Sushma Palimar: \textit{Computations in p-adic Discrete Dynamics and Real Quadratic Fields},
 Ph.D Thesis, 2012, National Institute of Technology Karnataka, Surathkal, INDIA 

\end{thebibliography}
\end{document}